\documentclass[twoside,a4paper,12pt]{article}
\usepackage[top=4.5cm, bottom=4.5cm, inner=3.5cm,outer=3.5cm]{geometry}
\usepackage[utf8]{inputenc}
\usepackage{amsthm}
\usepackage{mathtools}
\usepackage{csquotes}
\usepackage{amsmath}
\usepackage[new]{old-arrows}
\usepackage{graphicx}
\usepackage{dsfont}
\usepackage{soul}
\usepackage{amssymb}
\usepackage{multicol}
\usepackage{courier}
\usepackage{listings}
\usepackage{eurosym}
\usepackage{subfiles}
\usepackage{multicol}
\usepackage{pdfpages}
\usepackage{footnote}
\usepackage{color}
\definecolor{myblue}{rgb}{.25, .25, .9}
\definecolor{myred}{rgb}{.3, .3, .3}
\definecolor{myred2}{rgb}{.3, .3, .3}
\definecolor{mygreen}{rgb}{.25, .6, .5}
\usepackage{adjustbox}
\usepackage[colorlinks=true,
            linkcolor=myred2,
            urlcolor=myred,
            citecolor=myred]{hyperref}
\usepackage{pdfpages}
\usepackage[english]{babel}
\usepackage{array}
\usepackage{enumerate}
\usepackage{gensymb}
\usepackage[nottoc]{tocbibind}
\usepackage{lettrine}
\usepackage[normalem]{ulem}
\usepackage[font=small,labelfont=bf]{caption}
\usepackage{cancel}
\usepackage{nomencl}
\usepackage{mathrsfs}
\usepackage{changepage}
\usepackage{amsmath,stmaryrd,graphicx,float}
\numberwithin{equation}{section}

\usepackage{tikz-cd}
\usetikzlibrary{arrows}

\usepackage{subcaption} 

\usepackage{titlesec}

\newtheoremstyle{mystyle}
  {}
  {}
  {\itshape}
  {}
  {\bfseries}
  {.}
  { }
  {}

\theoremstyle{mystyle}

\newtheorem{theorem}{Theorem}[section]

\newtheorem{definition}[theorem]{Definition}
\newtheorem{lemma}[theorem]{Lemma}

\newtheorem{corollary}[theorem]{Corollary}

\newtheorem{example}[theorem]{Example}
\newtheorem{remark}[theorem]{Remark}

\newcommand\xqed[1]{%
  \leavevmode\unskip\penalty9999 \hbox{}\nobreak\hfill
  \quad\hbox{#1}}
\newcommand\exampleEnd{\xqed{$\circ$}}

\usepackage{etoolbox}

\DeclareMathAlphabet{\mathbbold}{U}{bbold}{m}{n}

\titleformat{\section}[runin]
{\bfseries}{\llap{\thesection\hskip 9pt}}{0pt}{}
\titlespacing*{\section}
{0pt}{3.25ex plus 1ex minus .2ex}{1.5ex plus .2ex}

\titleformat{\subsection}[runin]
{\bfseries}{\llap{\thesubsection\hskip 9pt}}{0pt}{}
\titlespacing*{\subsection}
{0pt}{3.25ex plus 1ex minus .2ex}{1.5ex plus .2ex}

\titleformat{\subsubsection}[runin]
{\bfseries}{\llap{\thesubsubsection\hskip 9pt}}{0pt}{}
\titlespacing*{\subsubsection}
{0pt}{3.25ex plus 1ex minus .2ex}{1.5ex plus .2ex}

\titleformat{\paragraph}[runin]
{\bfseries}{\llap{\theparagraph\hskip 9pt}}{0pt}{}
\titlespacing*{\paragraph}
{0pt}{3.25ex plus 1ex minus .2ex}{1.5ex plus .2ex}

\usepackage[titles]{tocloft}

\usepackage{fancyhdr}
\usepackage{lastpage}
\fancypagestyle{plain}{
\fancyhf{}
\fancyhead[LE]{\thepage\quad\textbf{\nouppercase\leftmark}}
\fancyhead[RO]{\textbf{\nouppercase\rightmark}\quad\thepage}
}

\fancypagestyle{basicstyle}{%
\fancyhf{}
\fancyhead[LE]{\thepage}
\fancyhead[RO]{\thepage}
}

\usepackage{algorithm}
\usepackage{algorithmic}[1]

\usepackage{chngcntr}
\counterwithin{figure}{section}

\makeatletter
\newbox\dottedarrow@box
\setbox\dottedarrow@box\hbox
  {%
    \begin{tikzpicture}
      \draw[dotted,->] (0,0) -- (1.5em,0);
    \end{tikzpicture}%
  }
\newcommand*\dottedarrow
  {\relax\ifmmode\expandafter\dottedarrow@m\else\expandafter\dottedarrow@t\fi}
\newcommand*\dottedarrow@t[1][1.5em]
  {\resizebox{#1}{!}{\raisebox{.5ex}{\usebox\dottedarrow@box}}}
\newcommand*\dottedarrow@m[1][]
  {%
    \if\relax\detokenize{#1}\relax
      \mathchoice
        {\dottedarrow@t}
        {\dottedarrow@t}
        {\dottedarrow@t[1.1em]}
        {\dottedarrow@t[0.9em]}%
    \else
      \dottedarrow@t[#1]%
    \fi
  }
\makeatother

\usepackage[backend=bibtex,
maxbibnames=99,
style=alphabetic,
bibencoding=ascii,
giveninits=true, 
]{biblatex}
\renewbibmacro{in:}{}

\usepackage{rotating}
\usepackage{mathtools}

\newcounter{savecntr}

\title{{\textbf{\Large On topological properties of closed attractors}}}
\author{Wouter Jongeneel\setcounter{savecntr}{\value{footnote}}\thanks{The author is with the KTH Royal Institute of Technology (DCS) and Digital Futures, Stockholm. This work was supported by Digital Futures through the project ``\textit{Programmability of Cells}'' and the author is grateful to Matthew D. Kvalheim and Esra Öncel for feedback. Contact: \texttt{wouterjo@kth.se}, website: \url{wjongeneel.nl}.}
}
\date{\small{\today}}

\begin{document}
\maketitle
\thispagestyle{empty}

\begin{abstract}
The notion of an attractor has various definitions in the theory of dynamical systems. Under compactness assumptions, several of those definitions coincide and the theory is rather complete. However, without compactness, the picture becomes blurry. To improve our understanding, we characterize in this work when a closed---not necessarily compact---asymptotically stable attractor on a locally compact metric space is a strong deformation retract of its domain of attraction. This enables a further structural study of feedback stabilization problems.   
\end{abstract}

{\footnotesize{
\noindent\textbf{\textit{Keywords}}|attractors, dynamical systems, stability\\
\textbf{\textit{AMS Subject Classification (2020)}}| 37B25, 55P10, 93D20
}} 


\section{Introduction}
The (topological) dynamical systems theory of (asymptotically stable) \textit{compact} attractors is rather mature, \textit{e.g.}, see~\cite{ref:bhatiahajek2006local,ref:bhatia1970stability,ref:conley1978isolated,ref:akin1993general}. When an attractor is merely \textit{closed}, however, our theory is significantly less complete and unified. In this work, we aim to contribute to improving our understanding here by answering the following question:\\\\
\textbf{{Question:}} ``\textit{When is a closed attractor $A$, on a metric space $(X,d)$, homotopy equivalent to its basin of attraction $B(A)$?}''
\\\\
We will make this question more precise in the remainder of the introduction, but first we elaborate on its relevance.

The study of topological relations between attractors and their domain of attraction is a classical one and of importance in feedback control theory. A reason being, due to insufficiently accurate models and computational obstructions, we are typically unable to find an explicit expression for the domain of attraction, especially when feedback comes to play. However, applications demand an understanding of this set, \textit{e.g.}, what are all the perturbations a cruising aeroplane, oscillating gene regulatory network or walking robot can recover from? Or differently put, given a control system and some desirable domain of attraction $B'$, can we find continuous feedback such that the resulting basin of attraction $B(A)$ equals $B'$ in some sense? If such a feedback does not exist, how to overcome this? The topological viewpoint provides for coarse, but general answers to these questions. A basic example to have in mind is that a pendulum cannot be globally stabilized upright using continuous feedback, essentially due to the circle $\mathbb{S}^1$ not being contractible. However, introducing a discontinuity in the feedback, that is, a single jump or switch, does allow for global stabilization. Topologically speaking, one needs to ``cut" the circle. With this article, we aim to contribute to this stream of research. 

We cannot do justice to the rich history of this line of work, but we highlight the seminal contributions \cite{ref:wilson1967structure}, \cite{ref:bhatia1970stability}, \cite{ref:gunther1993every}, \cite[Thm.~21]{ref:sontag2013mathematical}, ~\cite{BhatBernstein} and adjacent \cite{ref:brockett1983asymptotic,ref:krasnosel1984geometrical,ref:zabczyk1989,ref:Coron1990}. We highlight \cite{ref:mansouri2007local,ref:mansouri2010topological,ref:moulay2010topological,ref:bernuau2013retraction,ref:kvalheim2022necessary,
ref:yao2022topological2,ref:yao2023domain2,ref:kvalheim2022obstructions,ref:kvalheim2023relationships,ref:jongeneelECC24} as more recent contributions and we point the reader to \cite{ref:jongeneel2023topological} for a recent overview.

The early focus on \textit{compact} attractors can be understood as their study encapsulates equilibrium points and limit cycles \cite{ref:auslanderbhatiasiebert1967asymptotic}.
Nevertheless, the focus on closed, but non-compact attractors in particular, is also of great theoretical and practical interest. From a practical perspective, we may provide the following examples. 

\begin{enumerate}[(i)]
\item \textit{The kernel of output maps}. Suppose we have a continuous nonlinear control system of the canonical form $\dot{x}=f(x,u)$, $y=h(x)$, with $0\in \mathrm{dom}(h)$, and we would like to globally zero the output $y$ using an appropriate choice of static state-feedback $x\mapsto \mu(x)$ that enters as the input $u$. That means that we aim to render $h^{-1}(0)=\{x\,|\, h(x)=0\}$ a global attractor. Since $h$ is continuous, $h^{-1}(0)$ is closed, but not necessarily compact.
\item \textit{Synchronization and estimation}. In the context of, for instance, observer design on a state space $X$, we generally aim to stabilize diagonal sets of the form $\Delta_X=\{ (x,x)\,|\, x\in X\}$. If $X$ is not compact, so is $\Delta_X$.  
\item \textit{Time-varying systems}. Consider the time-varying ODE $\dot{x}=f(x,t)$ on a space $X$. Suppose we want to understand stability of the set $A\subseteq X$ under $f$. In that case, we could study stability of $A\times \mathbb{R}_{\geq 0}$ under the autonomous ODE $\dot{x}=f(x,s)$, $\dot{s}=1$. 
\item \textit{Overparametrized learning}. 
In the context of machine learning, the majority of problems are formulated as optimization problems of the form $\inf_{\theta \in \mathbb{R}^d}L(\theta)$, where $L:\mathbb{R}^d\to \mathbb{R}$ is some differentiable loss function. For instance, $L(\theta)=\frac{1}{n}\sum^n_{i=1} \ell(h(x_i,\theta),y_i)^2$, where $(x_i,y_i)$ are pairs of data points, $\ell$ is a (local) loss function and $h$ is a predictor parametrized by $\theta$, \textit{e.g.}, $x_i$ is an image and $y_i$ a binary classification variable that indicates if the image contains a dog or not. The map $h$ is usually a neural network and the optimal $\theta$ is typically sought through an approximation of $\dot{\theta}=-\nabla L(\theta)$. When the network is \textit{overparametrized}, that is, $d \gg n$, the set of optimizers can be shown to be closed and unbounded under appropriate assumptions \cite{ref:nguyen2019connected}. This means we study a closed attractor. 
\end{enumerate}  

From a theoretical perspective, closed attractors are challenging as compactness is a highly convenient property exploited in a lot of proofs, \textit{e.g.}, consult~\cite{ref:bhatiahajek2006local,ref:bhatia1970stability,ref:conley1978isolated,ref:akin1993general,ref:guillemin2010differential}, more concretely, see for instance \cite[Thm.~1]{ref:kvalheim2022necessary} for a result where several of these constructions come together.

To elaborate on our question, we will assume that $(X,d)$ is a locally compact metric space and that $A$ is uniformly asymptotically stable under some continuous dynamical system. Then, we will largely address our question using Borsuk's \textit{retraction theory}~\cite{ref:hu1965,ref:borsuk1967theory,ref:dydak2012ideas}. Moreover, we are particularly inspired  by Auslander's work towards unifying stability through filters \cite{ref:auslander1977filter}, that is, for compact attractors there is no difference between metric- and topological definitions of stability, but for closed attractors this difference is non-trivial and neatly captured by neighbourhood filters. In general, the line of work by Bhatia and coworkers \cite{ref:auslanderbhatiasiebert1967asymptotic,ref:bhatia1967asymptotic,ref:bhatiahajek2006local,ref:bhatia1970stability} provides us with the right foundations. We remark that their work builds upon the seminal monographs by Nemytskii and Stepanov~\cite{ref:nemytskii} and Zubov~\cite{ref:zubov1964methods}.

Other noteworthy developments are H{\'a}jek's para-stability \cite{ref:hajek1972ordinary} and Hurley's work on exploiting locally compact, $\sigma$-compact  spaces in the context of non-compact attractors under maps \cite{ref:hurley2001weak}. We highlight that if a space $X$ is locally compact and $\sigma$-compact, then there is a countable set of compact sets $K_1,K_2,\dots,$ such that $X=\cup_{i\in \mathbb{N}_{>0}}K_i$, with $K_i \subseteq \mathrm{int}\,K_{i+1}$ \cite[p. 94]{ref:bourbakigeneraltopology}. This structure is exploited in the majority of work concerned with closed attractors, \textit{e.g.}, below we assume that $(X,d)$ is not only locally compact, but also separable, this to appeal to \cite[Lem. V.4.26]{ref:bhatia1970stability}. Indeed, for metric spaces this is equivalent to assuming local compactness and $\sigma$-compactness, \textit{e.g.}, see \cite[Ch. V.1]{ref:nemytskii}. 

Although most works appeal to a metric structure, we highlight one exception. There, the price to pay is that assumptions on $A$ are arguably stronger. Specifically, building upon the likes of Zubov and Ura, Bhatia and H{\'a}jek provide a comprehensive theory for closed attractors $A$, with compact topological boundaries $\partial A$, under semi-dynamical systems on locally compact Hausdorff spaces \cite{ref:bhatiahajek2006local}. The compact boundary allows for a theory reminiscent of compact attractors, that is, one may focus on a compact subset of $(B(A)\setminus A) \cup \partial A$, thereby one can appeal to Urysohn's lemma (as compact Hausdorff spaces are normal) and construct a continuous Lyapunov function \cite[Thm. 10.6]{ref:bhatiahajek2006local}. A picture to have in mind is shown in Figure~\ref{fig:closedexamples} $(i)$. 

Although our focus is on topological dynamical systems, we highlight that on $\mathbb{R}^n$, and under different regularity assumptions (\textit{e.g.}, local Lipschitzness of inclusions), a fairly complete converse Lyapunov theory for closed attractors is available, even for control Lyapunov functions \cite{ref:kellett2004weak}, see also \cite{ref:lin1996smoothV2,ref:albertini1999continuous}. We also highlight that our focus is on \textit{homotopy} equivalence. This, to strike a balance between generality and distinctiveness. 
Nonetheless, there is an interesting line of work on \textit{shape} equivalence, which would be more general but less distinctive,~\textit{e.g.}, see~\cite{ref:hastings1979higher,ref:garay1991strong,ref:gunther1993every,ref:kapitanski2000shape,ref:giraldo2001some,ref:giraldo2009singular} and \cite[Prop.~1]{ref:kvalheim2022necessary}.

In Section~\ref{sec:prelim} we provide the background material on topology and dynamical systems, plus we develop a few new tools. In Section~\ref{sec:correct:Wilson} we introduce our running example(s) and in Section~\ref{sec:main} we detail and prove our main results, that is, we characterize in Theorem~\ref{thm:main:cofib} when $A$ is a strong deformation retract of $B(A)$. We close the work in Section~\ref{sec:conclusion}

    \begin{figure}
        \centering
        \includegraphics[scale=1]{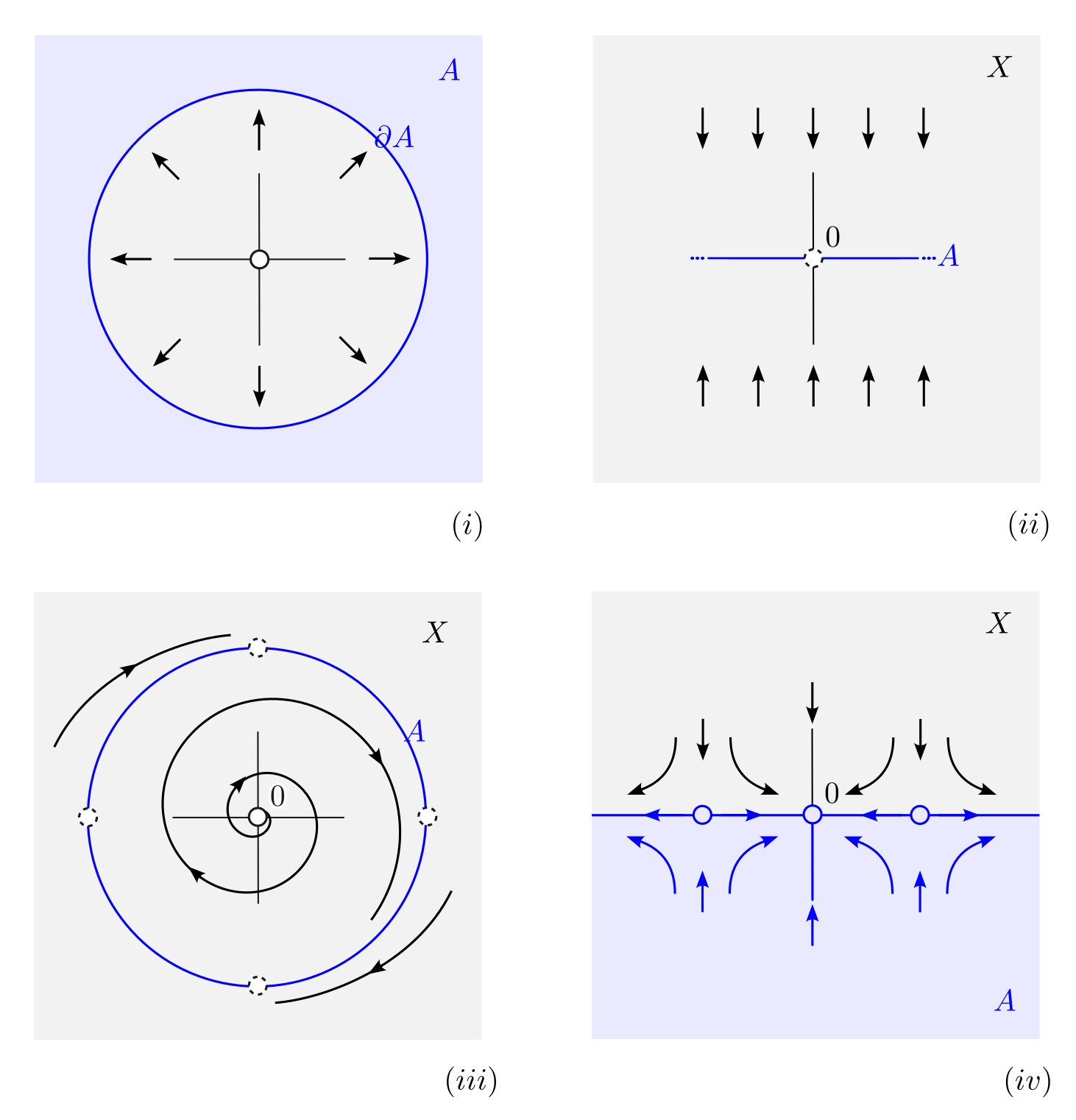}
        \caption{Examples of a closed, but non-compact, attractors, with in $(i)$ $A$ being of the form $\mathbb{R}^2\setminus \{x:\|x\|_2<1\}$ such that $\partial A =\mathbb{S}^1$, whereas in $(ii)$ the underlying space $X$ is of the form $\mathbb{R}^2\setminus \{0\}$ such that $A$ as drawn is closed. In $(iii)$ $X$ is $\mathbb{R}^2$ with $\{ (1,0),(0,1),(-1,0),(0,-1)\}$ removed, this is slightly more involved version of the example in Section~\ref{sec:correct:Wilson}. At last, in $(iv)$ one sees how closed attractors might emerge by grouping several invariant sets.}
        \label{fig:closedexamples}
    \end{figure}

\section{Preliminaries and notation}
\label{sec:prelim}
In this section we introduce all the topology and dynamical systems theory to define and prove our main result. 

\subsection{General topology}
We will work with a \textbf{\textit{metric space}} $(X,d)$, \textit{e.g.}, see \cite[Ch. 3-7]{ref:munkrestopology}. In particular, this means we work with the topology $\tau$ induced by $d$, \textit{i.e.}, open metric balls of the form $B_{r}(x;d) := \{x'\in X \, | \, d(x,x')<r\}$. Now, given a closed subset $A\subseteq X$, consider for some $\varepsilon>0$ the sets
\begin{align*}
&N_{\varepsilon}(A;d) := \{x\in X\,|\, d(x,A)< \varepsilon \}\quad \text{and}\\
&D_{\varepsilon}(A;d) := \{x\in X\,|\, d(x,A)\leq \varepsilon \}, 
\end{align*}
where $d(\cdot,A):X\to \mathbb{R}_{\geq 0}$ is defined through
\begin{equation*}
x\mapsto d(x,A) := \inf_{x'\in  A}d(x,x'). 
\end{equation*}
It is convenient to recall that $x\mapsto d(x,A)$ is $1$-Lipschitz. Also, if $d$ is irrelevant or clear from the context, we drop it in the notation, \textit{e.g.}, we write $N_{\varepsilon}(A)$. 

\subsection{Retraction theory}
\label{sec:retraction}
Two continuous maps $f,g:X\to Y$ are \textit{homotopic}, denoted $f\simeq_h g$, when there is a continuous map $H:X\times [0,1]\to Y$ such that $H(\cdot,0)=f$ and $H(\cdot,1)=g$.
Two topological spaces $X$ and $Y$ are said to be \textbf{\textit{homotopy equivalent}} when there are continuous maps $f:X\to Y$ and $g:Y\to X$ such that $f\circ g\simeq_h \mathrm{id}_Y$ and $g\circ f\simeq_h \mathrm{id}_X$. We will overload notation and also write $X\simeq_h Y$. In the context of dynamical systems, we aim to understand when some attractor $A$ and its basin of attraction $B(A)$ are homotopy equivalent. In this case, one is naturally drawn to \textit{retractions}.     

A set $A\subseteq X$ is a \textbf{\textit{retract}} of $X$ when there is a continuous map $r:X\to A$ such that $r\circ \iota_A = \mathrm{id}_A$, for $\iota_A$ the inclusion map $\iota_A:A\hookrightarrow X$. Note, if $r:X\to A$ is a retract, then $A$ is closed if $X$ is a Hausdorff space. Another convenient fact is that if we find a $U$ such that $A\subseteq U\subseteq X$, then $r|_U : U \to A$ is also a retract. The set $A$ is said be a \textbf{\textit{deformation retract}} of $X$ when $A$ is a retract and additionally $\iota_A\circ r \simeq_h \mathrm{id}_X$, implying that $X$ is homotopy equivalent to $A$. When, additionally, the homotopy is \textit{stationary relative to $A$}, we speak of a \textbf{\textit{strong deformation retract}}. We emphasize that this terminology is not completely agreed upon \textit{cf}.~\cite[Ch.~0]{Hatcher} and~\cite[Sec.~1.11]{ref:hu1965}. 

As an important intermediate notion, a set $A\subseteq X$ is said to be a \textbf{\textit{weak deformation retract}} of $X$ when every open neighbourhood $U$ of $A$ contains a strong deformation retract $V\supseteq A$ of $X$. We emphasize that these definitions rely on the topology on $X$.  

Suppose for the moment that $X$ is a locally compact metric space. Then, in \cite[Thm.~5]{ref:moulay2010topological} it was shown that compact (asymptotically stable) attractors are weak deformation retracts of $B(A)$. To exploit this result, we need another notion of retraction.    

A set $A\subseteq X$ is a \textbf{\textit{neighbourhood retract}} of $X$ when there is an open neighbourhood $U\subseteq X$ of $A$ such that $A$ is a retract of $U$. This definition extends naturally to (strong) deformation retracts. We emphasize again that this is a topological definition, with some variation throughout the literature.  

Now, if $B(A)$ weakly deformation retracts onto $A$, while $A$ is a neighbourhood deformation retract of $B(A)$, we have by composition that $A\simeq_h B(A)$. This is the philosophy as set forth in, for instance, \cite{ref:kapitanski2000shape,ref:moulay2010topological}. Then, by leveraging this approach and by appealing to cofibrations, we answered the research question of this article, but for compact attractors on locally compact Hausdorff spaces, in \cite{ref:jongeneelECC24}.

Prior to \cite{ref:jongeneelECC24}, only sufficient conditions were known, \textit{e.g.}, for $A$ being a smooth submanifold. 
In this work, we complete this line of work for closed attractors on metric spaces. The crux is to generalize these notions of retraction to \textit{neighbourhood filters}.

\subsection{Retractions and neighbourhood filters} 
\label{sec:retract:and:filters}
Let $X$ be a topological space, a \textbf{\textit{filter}} $\mathscr{F}$ on a $X$ is collection of subsets of $X$ such that (i) $\varnothing \notin \mathscr{F}$; (ii) if $U\in \mathscr{F}$ and $V\supseteq U$ then $V\in \mathscr{F}$; and (iii) if $U,V\in \mathscr{F}$ then $U\cap V\in \mathscr{F}$ \cite[Ch. I.6]{ref:bourbakigeneraltopology}. Note that (i) and (ii) together imply that $X\in \mathscr{F}$. Filters are convenient to study convergence, in general~\cite[Ch. I.7]{ref:bourbakigeneraltopology}, and the remaining subsection should be understood in that light. Then, for two filters $\mathscr{F}_1$ and $\mathscr{F}_2$ on $X$ we write $\mathscr{F}_1\leq \mathscr{F}_2$ when $\mathscr{F}_2$ is finer than $\mathscr{F}_1$, that is, when $U\in \mathscr{F}_1\implies \exists V\in \mathscr{F}_2:V\subseteq U$. 
 
Now, let $(X,d)$ be a metric space with $A\subseteq X$ some closed subset, then two important filters for us are as follows.  
\begin{enumerate}[(i)]
\item The \textit{topological} neighbourhood filter of $A$, denoted $\mathscr{F}_{\tau}$, and defined through $\mathscr{F}_{\tau}:=\{U\subseteq X\,|\, U$ is an open neighbourhood of $A\}$. 
\item The \textit{metric} neighbourhood filter of $A$, denoted $\mathscr{F}_d$, and defined through $\mathscr{F}_{d}:=\{U\subseteq X\,|\, U\supseteq N_{\varepsilon}(A;d)$ for some $\varepsilon>0 \}$.  
\end{enumerate}
The reason being, these two filters capture the different notions of convergence as seen in the study of attractors, that is, purely topological, or with respect to some metric.    

It readily follows that if $U\in \mathscr{F}_d$, then $U\in \mathscr{F}_{\tau}$ and so $\mathscr{F}_d\leq \mathscr{F}_{\tau}$, however, the converse fails to be true in general (\textit{e.g.}, see Example~\ref{ex:S1closed}). In general, we speak of a neighbourhood filter $\mathscr{F}$, with respect to some understood set $A\subseteq X$, when all elements of $\mathscr{F}$ are neighbourhoods of $A$. Indeed, $\mathscr{F}_{\tau}$ is the finest neighbourhood filter.    

Using filters, we generalize several notions of retraction. 

\begin{definition}[$\mathscr{F}$-weak deformation retract]
\label{def:F:weak:def:retract}
Let $\mathscr{F}$ be a neighbourhood filter on $X$ with respect to $A\subseteq X$.
Then, $A$ is a $\mathscr{F}$-weak deformation retract of $X$ when every $U\in \mathscr{F}$ contains a strong deformation retract of $X$. 
\end{definition}  
 
\begin{definition}[$\mathscr{F}$-neighbourhood (deformation) retract]
\label{def:F:neighbourhood:retract}
Let $\mathscr{F}$ be a neighbourhood filter on $X$ with respect to $A\subseteq X$.
Then, $A$ is a $\mathscr{F}$-neighbourhood (deformation) retract of $X$ when there is a $W\in \mathscr{F}$ that (deformation) retracts onto $A$. 
\end{definition} 

One may generalize Definition~\ref{def:F:neighbourhood:retract} to strong neighbourhood deformation retracts in the obvious way. One may also observe a certain type of duality, as captured by the following two lemmas. 
\begin{lemma}[Coarser and finer filter retracts]
Let $\mathscr{F}$ and $\mathscr{G}$ be neighbourhood filters on $X$, with respect to $A\subseteq X$, such that $\mathscr{F}\leq \mathscr{G}$.
\begin{enumerate}[(i)]
\item If $A$ is a $\mathscr{G}$-weak deformation retract of $X$, then $A$ is a $\mathscr{F}$-weak deformation retract of $X$.
\item If $A$ is a $\mathscr{F}$-neighbourhood deformation retract of $X$, then $A$ is a $\mathscr{G}$-neighbourhood deformation retract of $X$.
\end{enumerate}
\end{lemma}
\begin{proof}
(i) Since $\mathscr{F}\leq \mathscr{G}$, $U\in \mathscr{F}\implies \exists V\in \mathscr{G}:V\subseteq U$, any $U\in \mathscr{F}$ contains a strong deformation retract of $X$, namely, a subset of $V\in \mathscr{G}$. 

(ii) Since there is a $W\in \mathscr{F}$ that deformation retracts onto $A$ and $\mathscr{F}\leq \mathscr{G}$, there is always a $V\in \mathscr{G}$ such that $V\subseteq W$ and thus $W$ must be in $\mathscr{G}$ by superset completion (filter property (ii)).   
\end{proof}

\begin{lemma}[A strong deformation retract factored through filters]
\label{lem:strong:def:retract:filter}
Suppose we work with filters that are neighbourhood filters on $X$ with respect to $A\subseteq X$.

\begin{enumerate}[(i)]
\item Let $\mathscr{F}$ be a filter such that $A$ is a strong $\mathscr{F}$-neighbourhood deformation retract. Then, $A$ is a strong deformation retract of $X$ if and only if there is a filter $\mathscr{F}'$ such that $\mathscr{F}\leq \mathscr{F}'$ and $A$ is a $\mathscr{F}'$-weak deformation retract of $X$.
\item Let $\mathscr{G}$ be a filter such that $A$ is a $\mathscr{G}$-weak deformation retract. Then, $A$ is a strong deformation retract of $X$ if and only if there is a filter $\mathscr{G}'$ such that $\mathscr{G}'\leq \mathscr{G}$ and $A$ is a strong $\mathscr{G}'$-neighbourhood deformation retract of $X$.
\end{enumerate}

\end{lemma}
\begin{proof}
(i) Suppose there is a $\mathscr{F}'$ that complies with the statement from the lemma, then, any $U\in \mathscr{F}'$ contains a strong deformation retract of $X$. As $\mathscr{F}'$ is finer than $\mathscr{F}$, there is always a $U$ contained in the neighbourhood $V\in \mathscr{F}$ that strongly deformation retracts onto $A$. Therefore, $X$ strongly deformation retracts onto $A$, \textit{e.g.}, by composition. 

Now suppose that $X$ strongly deformation retracts onto $A$, while $A\subseteq X$ is a strong $\mathscr{F}$-neighbourhood deformation retract. Clearly, $A$ is a weak deformation retract of $X$ in the standard sense, thus, $A$ is a $\mathscr{F}_{\tau}$-weak deformation retract of $X$; and $\mathscr{F}_{\tau}$ is the finest neighbourhood filter on $X$. 

(ii) Similar to (i), suppose that such a $\mathscr{G}'$ exists. Any $U\in \mathscr{G}$ contains a strong deformation retract of $X$ and since $\mathscr{G}'\leq \mathscr{G}$, we can find for any $V\in \mathscr{G}'$ a $U$ such that $U\subseteq V$. Doing this for the $V$ that strongly deformation retracts onto $A$ concludes this step. 

Now suppose that $X$ strongly deformation retracts onto $A$, then $X$ is a trivially a strong neighbourhood deformation retract of $A$. In particular, the coarsest (trivial) filter $\mathscr{G}':=\{X\}$ satisfies the requirements of the lemma.    
\end{proof}

We end this subsection by recalling that if $A$ is compact, then, $\mathscr{F}_{\tau}$ and $\mathscr{F}_d$ are equivalent, which is a well-known result, \textit{e.g.}, see \cite[\S 27]{ref:munkrestopology}. One may also note that when $A$ is compact, $\mathscr{F}_d$ is cofinal in $\mathscr{F}_{\tau}$, \textit{e.g.}, see \cite[Prop. 5.1]{ref:gobbino2001topological}. To keep the work remotely self-contained, we collect a proof. 
\begin{lemma}[On $\varepsilon$-neighbourhoods]
\label{lem:eps:nbhd:U}
Let $A$ be a compact subset of the metric space $(X,d)$. For any open neighbourhood $U$ of $A$, there is a $\varepsilon>0$ such that $N_{\varepsilon}(A;d)\subseteq U$.
\end{lemma}
\begin{proof}
If $A=X$, we are done, so suppose $A\neq X$. Thus, pick any open neighbourhood $U$ of $A$, then $X\ni x\mapsto d(x,X\setminus U)\in \mathbb{R}_{\geq 0}$ is continuous. Now, since $A$ is compact (and thus closed as $X$ is Hausdorff), $X\setminus U$ is closed, $A\cap (X\setminus U) = \varnothing$ and $\varepsilon:=\min_{x\in A}d(x,X\setminus U)>0$. At last, to show that $N_{\varepsilon}(A;d)\subseteq U$, suppose it is not, then, there is a point $x'\in X\setminus U$ such that $d(x',A)<\varepsilon$, contradicting the definition of $\varepsilon$.  
\end{proof}

Lemma~\ref{lem:eps:nbhd:U} is precisely the reason why for compact attractors several definitions of stability are equivalent, \textit{e.g.}, using $\mathscr{F}_d$ or $\mathscr{F}_{\tau}$. 

To further illustrate this, suppose that $A=\{0\}$ is Lyapunov stable under some smooth ODE on $(\mathbb{R}^n,d(x,y):=\|x-y\|_2)$, say $\dot{x}=f(x)$ and let $\varphi$ be the corresponding flow. In this case, $N_{\varepsilon}(A;d)=B_{\varepsilon}(0;d)$. The standard ``\textit{topological}'' definition of Lyapunov stability of $A$ under $\varphi$ would say that for any open neighbourhood $U\ni 0$ there is another open neighbourhood $V\ni 0$ such that $\{\varphi^t(x)\,|\, x\in V,\, t\geq 0\}\subseteq U$. Clearly, this must be true for $U:=B_{\varepsilon}(0;d)$, but then by Lemma~\ref{lem:eps:nbhd:U}, we can find some $\delta(\varepsilon)>0$ such that $B_{\delta(\varepsilon)}(0;d)\subseteq V$ and thus $\{\varphi^t(x)\,|\, x\in B_{\delta(\varepsilon)}(0;d),\, t\geq 0\}\subseteq B_{\varepsilon}(0;d)$, which would be the ``\textit{metrical}'' definition of Lyapunov stability. For the other direction, given some open neighbourhood $U$, again by Lemma~\ref{lem:eps:nbhd:U} we can find a $\varepsilon>0$ such that $B_{\varepsilon}(0;d)\subseteq U$, then set $V:=B_{\delta(\varepsilon)}(0;d)$.   

If we relax compact to closed, Lemma~\ref{lem:eps:nbhd:U} fails to be true in general. 

\begin{example}[Open neighbourhoods of closed subsets 1]
\label{ex:open:nbhd:closed:1}
\upshape{
Let the locally compact metric space $(X,d)$ be given by $X:=\mathbb{R}^2$ and $d(x,y):=\|x-y\|_2$. Now consider $A:=\{(x_1,x_2)\in \mathbb{R}^2\,|\, x_2 = 0\}$, which is a closed but non-compact subset. Then, construct the open neighbourhood $U:=\{(x_1,x_2)\in \mathbb{R}^2\,|\, |x_2| < e^{-x_1^{2}}\}$ of $A$. Clearly, there is no $\varepsilon>0$ such that $N_{\varepsilon}(A;d)=\{(x_1,x_2)\in \mathbb{R}^2\,|\,|x_2|< \varepsilon \}$ is contained in $U$, see Figure~\ref{fig:xaxisclosed} $(ii)$.    
\exampleEnd 
}
\end{example} 

    \begin{figure}
        \centering
        \includegraphics[scale=1]{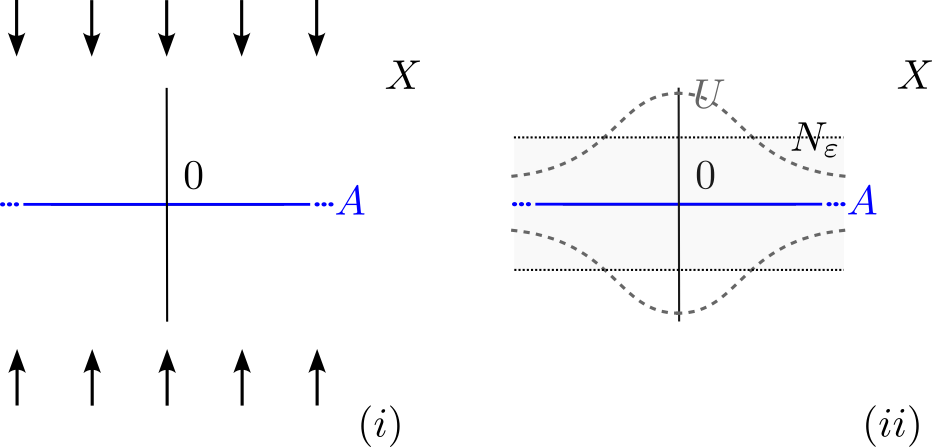}
        \caption{Example~\ref{ex:open:nbhd:closed:1}: for closed, but non-compact, subsets $A$ of a metric space $X$ it is not true that any open neighbourhood $U$ of $A$ contains a metric neighbourhood $N_{\varepsilon}$ of $A$.}
        \label{fig:xaxisclosed}
    \end{figure}

\subsection{Cofibrations}
Previously, we used the notion of a \textit{cofibration} to capture when a compact attractor $A$ is a strong deformation retract of its basin of attraction $B(A)$ \cite{ref:jongeneelECC24}. In particular, cofibrations are closely related to neighbourhood deformation retracts. As before, we need to adjust these notions for filters. 

To define cofibrations, we need the following. 
Let $X$ be a topological space and $A\subseteq X$, then, a pair $(X,A)$ has the \textit{homotopy extension property} (HEP) when, for any $Y$, the diagram (arrows are continuous maps) 
\begin{equation}
\label{equ:diag1}
\begin{tikzcd}[row sep=normal, column sep=large]
(A\times [0,1]) \cup (X \times \{0\}) \arrow{r}{} \arrow[hookrightarrow]{d} & Y  \\
X\times [0,1] \arrow[dotted, bend right]{ur}
\end{tikzcd}
\end{equation}
can always be completed (``completed'' means that the \textit{dotted arrow} \dottedarrow\, can be found) to be commutative. Thus, more explicitly, given a a homotopy $H:A\times [0,1]\to Y$ and some map $g:X\to Y$ such that $H(\cdot,0)=g|_A$, one needs to be able to extend the homotopy from $A$ to $X$.
Pick $Y=(A\times [0,1]) \cup (X \times \{0\})$, then we see that $(X,A)$ having the HEP implies that $(A\times [0,1]) \cup (X \times \{0\})$ is a retract of $X\times [0,1]$. On the other hand, one can show that the existence of such a retract implies that $(X,A)$ has the HEP, that is, these two notions are equivalent, see Theorem~\ref{thm:cofib:retract:ndr}. 

Then, a continuous map $i:A\to X$ is said be a \textbf{\textit{cofibration}} if the following commutative diagram 
\begin{equation}
\label{equ:diag2}
\begin{tikzcd}[row sep=normal,column sep=large]
A \times \{0\} \arrow[d,"i\times \mathrm{id}"] \arrow[hookrightarrow]{r}               & A\times [0,1] \arrow{d} \arrow[rdd, bend left,"F"] &   \\
X\times \{0\} \arrow[hookrightarrow]{r} \arrow[rrd, bend right,"f"] & X\times [0,1] \arrow[rd, dotted, "\widetilde{F}"]               &   \\
                                    &                                    & Y
\end{tikzcd}
\end{equation}
can be completed for any triple $(f,F,Y)$, \textit{i.e.}, we can find $\widetilde{F}$.


We have introduced the notation $(f,F,\widetilde{F})$ so that we can easily define what we call a $\mathscr{F}$-cofibration.

\begin{definition}[$\mathscr{F}$-cofibration]
\label{def:F:cofib}
Let $A\subseteq X$ be closed and let $\mathscr{F}$ be some neighbourhood filter of $A$. The inclusion $\iota_A : A \hookrightarrow X$ is a $\mathscr{F}$-cofibration when there is a $U\in \mathscr{F}$ such that~\eqref{equ:diag2} can be completed, for any triple $(f,F,Y)$, with additionally $\widetilde{F}$ satisfying $\widetilde{F}(U\times \{1\}) \subseteq  F(A\times (0,1])$.
\end{definition}
If $\iota_A : A \hookrightarrow X$ is a $\mathscr{F}$-cofibration and $\mathscr{F}\leq \mathscr{F}'$ then $\iota_A : A \hookrightarrow X$ is also a $\mathscr{F}'$-cofibration. Indeed, below we clarify that a cofibration is a $\mathscr{F}_{\tau}$-cofibration. It turns out that we need $\mathscr{F}_d$-cofibrations to capture the right notion of convergence to be able to answer our question.

Next, we need another slight variation of the aforementioned notions of retraction, that of a \textit{neighbourhood deformation retract pair} (NDR pair).

\begin{definition}[NDR pair]
\label{def:NDR}
A pair $(X,A)$ is said to be an NDR pair if:
\begin{enumerate}[(i)]
    \item there is a continuous map $u:X\to [0,1]$ such that $A=u^{-1}(0)$; and
    \item there is a homotopy $H:X\times [0,1]\to X$ such that $H(x,0)=x$ for all $x\in X$, $H(a,s)=a$ for all $(a,s)\in A\times [0,1]$ and $H(x,1)\in A$ if $u(x)<1$.
\end{enumerate} 
\end{definition}

Importantly, in Definition~\ref{def:NDR} one can restrict the homotopy to the ``\textit{deformation}'' $H:W\times [0,1]\to X$, for $W:=u^{-1}([0,1))$. See \cite[ch. IV-VII]{ref:hu1965} for more on deformations. See that if $u(x)<1$ $\forall x\in X$, then, $A$ is a strong deformation retract of $X$. In general, however, we cannot assume $u$ to be of this form, see also \cite[App.]{ref:mayloopspaces1972} for the notion of a \textit{strong} NDR pair. See that for $(X,A)$ to be an NDR pair, $A$ must be closed. Now, a useful result is the following.
\begin{theorem}[{\cite[Ch. VII]{ref:bredon1993topology}}, {\cite[Ch.~6]{may1999concise}}]
\label{thm:cofib:retract:ndr}
    Let $A$ be closed in $X$, then, the following are equivalent:
    \begin{enumerate}[(i)]
        \item the inclusion $\iota_A:A\hookrightarrow X$ is a cofibration;
        \item \label{item:retract} $(A\times [0,1])\cup (X\times\{0\})$ is a retract of $X\times [0,1]$;
        \item \label{item:NDR} $(X,A)$ is an NDR pair. 
    \end{enumerate}
\end{theorem}

To capture that we want a specific deformation, we also adapt the notion of an NDR pair to a neighbourhood filter. The definition emphasizes the deformed neighbourhood $U$.  
\begin{definition}[$\mathscr{F}$-NDR pair]
\label{def:F:NDR}
A pair $(X,A)$ is said to be an $\mathscr{F}$-NDR pair when there is a $U\in \mathscr{F}$ and a continuous map $u:X\to [0,1]$ such that 
\begin{enumerate}[(i)]
    \item $A=u^{-1}(0)$;
    \item $u(x)=1$ for all $x\in X\setminus U$; and
    \item there is a homotopy (deformation) $H:U\times [0,1]\to X$ such that $H(u,0)=u$ for all $u\in U$, $H(a,s)=a$ for all $(a,s)\in A\times [0,1]$ and $H(u,1)\in A$ for all $u\in U$.
\end{enumerate} 
\end{definition}
See that if $(X,A)$ is a $\mathscr{F}$-NDR pair, then there is a $U\in \mathscr{F}$ such that $U\to A$ is a retract, \textit{i.e.}, for $H$ the homotopy from Definition~\ref{def:F:NDR}, select $H(\cdot,1)$. Differently put, $A$ is a $\mathscr{F}$-neighbourhood retract of $X$. 

\begin{remark}[$\mathscr{F}_d$-NDR pairs and metric neighbourhoods]
\label{rem:Fd:NDR}
\upshape{
Observe that if $(X,A)$ is an $\mathscr{F}_d$-NDR pair, then since there is $\varepsilon>0$ such that $N_{\varepsilon}(A;d)\subseteq U$, we can redefine $u$ to be $x\mapsto u(x) := \min \{1,\varepsilon^{-1} d(x,A)\}$, that is, we collapse $U$ to $N_{\varepsilon}(A;d)$. We exploit this below. 
\exampleEnd
}
\end{remark}

The next result is an adjustment of known results regarding the characterization of cofibrations (\textit{e.g.}, Theorem~\ref{thm:cofib:retract:ndr}), specialized to $\mathscr{F}_d$-cofibrations.

\begin{theorem}[$\mathscr{F}_d$-cofibration]
\label{thm:F:cofib:retract:ndr}
    Let $A$ be closed in $X$, then, the following are equivalent:
    \begin{enumerate}[(i)]
        \item the inclusion $\iota_A:A\hookrightarrow X$ is a $\mathscr{F}_d$-cofibration;
        \item \label{item:F:retract} $r:X\times [0,1]\to (A\times [0,1])\cup (X\times\{0\})$ is a retract with $r(V\times \{1\})\subseteq A\times (0,1]$ for some $V\in \mathscr{F}_d$; and 
        \item \label{item:F:NDR} $(X,A)$ is an $\mathscr{F}_d$-NDR pair. 
    \end{enumerate} 
\end{theorem}
\begin{proof}
We can largely follow \cite[Thm. VII.1.5]{ref:bredon1993topology}.   

$(ii)\implies (iii)$. We assume that $X\neq A$, otherwise the result is trivial. Let $r$ be the retraction at hand, as $r(V\times \{1\})\subseteq A \times (0,1]$ for some $V\in \mathscr{F}_d$, we have for $s:=r(\cdot,1)$ that $V\subseteq s^{-1}(A\times (0,1])=:U$ and thus, $U\in \mathscr{F}_d$ as $V\in \mathscr{F}_d$. Then, one readily sees that $H := \pi_X \circ r : U \times [0,1]\to X$ checks out as a homotopy in Definition~\ref{def:F:NDR}. Also, one can select $X\ni x\mapsto u(x) := \max_{t\in [0,1]}|t-\pi_{[0,1]}(r(x,t))|$ and see that this map $u$ satisfies item $(i)$ and $(ii)$ of Definition~\ref{def:F:NDR}. To see that $u$ is continuous, one may appeal to the subbase argument as in \cite[p. 432]{ref:bredon1993topology}. 

$(iii)\implies (ii)$. We assume that $X\neq A$, otherwise the result is trivial. The result follows from \cite[p. 432]{ref:bredon1993topology}. However, we emphasize that one does not readily have $V=U$ for $U\in \mathscr{F}_d$ from the $\mathscr{F}_d$-NDR condition. To be precise, in line with Remark~\ref{rem:Fd:NDR}, we may assume that $U=N_{\varepsilon}(A;d)$ for some $\varepsilon>0$. Then, following \cite[p. 432]{ref:bredon1993topology} we see that to get the desired map $r$, or $s=r(\cdot,1)$ for that matter, we need to consider the set $\{x\in U\,|\, 1-2u(x)>0\}$, which is precisely $W:=N_{(1/2)\varepsilon}(A;d)$, that is, we get that $s^{-1}(A\times (0,1])=W\in \mathscr{F}_d$.    

$(i)\implies (ii)$. Let $U\in \mathscr{F}_d$ be the neighbourhood of $A$ from Definition~\ref{def:F:cofib}. Pick $Y:= (A\times [0,1])\cup (X\times \{0\})$ and let $f$ and $F$ be inclusion maps. Then, $\widetilde{F}$ is the retraction $r$.    

$(ii)\implies (i)$. Let $r$ be the retraction and set $\widetilde{F}:=(f\cup F)\circ r$, for some $Y$. It follows that $\widetilde{F}$ factors as in Definition~\ref{def:F:cofib}.  
\end{proof}

\begin{remark}[Homotopy equivalence]
\label{rem:homotopic:factor}
\upshape{
Through the equivalence with $\mathscr{F}_d$-NDR pairs, we see that $A\hookrightarrow X$ being a $\mathscr{F}_d$-cofibration implies the existence of a $U\in \mathscr{F}_d$ such that for the retract $r:U\to A$ and inclusion maps $\iota_A:A\hookrightarrow X$ and $\iota_U : U \hookrightarrow X$ we have the homotopy equivalence $\iota_ U \simeq_h \iota_A \circ r$ as maps from $U$ to $X$. This enforces a relation between topological properties of the triple $(A,U,X)$. On the other hand, see that if $A$ is a $\mathscr{F}_d$-neighbourhood deformation retract of $U$, we have $\mathrm{id}_U \simeq_h \iota_A \circ r$, for $\iota_A : A \hookrightarrow U$.   
\exampleEnd
}
\end{remark}

\subsection{Dynamical systems}
We will largely follow \cite{ref:bhatiahajek2006local,ref:bhatia1970stability}. We study \textbf{\textit{semi-dynamical systems}} comprised of the triple $(X,\mathbb{R}_{\geq 0},\varphi)$. Here,  $X$ is a metric space and $\varphi:X\times \mathbb{R}_{\geq 0} \to X$ is a (global) semi-flow, that is, a map that satisfies for any $x\in X$:
\begin{enumerate}[(i)]
    \item $\varphi(x,0)=x$ (the \textit{initial value} axiom); 
    \item $\varphi(\varphi(x,t),s)=\varphi(x,t+s)$ $\forall s,t \in \mathbb{R}_{\geq 0}$ (the \textit{semi-group} axiom); and
    \item $\varphi$ is continuous (the \textit{continuity} axiom).
\end{enumerate}
For instance, $\varphi$ might correspond to a differential equation $\dot{x}=f(x)$ on $X$. We will usually write $\varphi^t$ instead of $\varphi(\cdot,t)$. If we could work with $\mathbb{R}$ instead of $\mathbb{R}_{\geq 0}$, we would speak of a flow and $\varphi^t$ would be a homeomorphism.

\subsection{Stability}
Given a metric space $(X,d)$ and some semi-dynamical system $(X,\mathbb{R}_{\geq 0},\varphi)$, we will be concerned with stability of a closed subset $A\subseteq X$ under this system.   

Specifically, we are concerned with the following stability notions (always understood to be with respect to a system $(X,\mathbb{R}_{\geq 0},\varphi)$). First, the set $A$ is said to be \textbf{\textit{uniformly stable}} when for each $\varepsilon>0$ there is a $\delta(\varepsilon)>0$ such that
 $\{\varphi^t(x)\,|\, x\in N_{\delta(\varepsilon)}(A;d),\,t\geq 0 \}\subseteq N_{\varepsilon}(A;d)$. 

Second, $A$ is said to be a \textit{\textbf{uniform attractor}} if there is some $a>0$, and for any $\varepsilon>0$ a $T(\varepsilon)>0$ such that  $\{\varphi^t(x)\,|\, x\in N_{a}(A;d),\,t\geq T(\varepsilon) \}\subseteq N_{\varepsilon}(A;d)$. 

Then, a closed set $A\subseteq X$ is said to be \textit{\textbf{uniformly asymptotically stable}} when it is both uniformly stable and a uniform attractor. We denote the corresponding basin of attraction by $B_u(A)$. We emphasize that these notions are \textit{uniform} in the sense that we work with neighbourhoods of $A$, not neighbourhoods of points. We also emphasize that these stability notions are not purely topological, they rely on the interplay between the metric $d$ and the set $A$, that is, $N_{\varepsilon}(A;d),N_{\delta(\varepsilon)}(A;d)$ and $N_{a}(A;d)$ are all elements from the neighbourhood filter $\mathscr{F}_d$, we are not selecting any element from $\mathscr{F}_{\tau}\setminus \mathscr{F}_d$. 

We remark that early work defined (Lyapunov) stability exclusively through metrics, in fact, through norms \textit{e.g.}, see \cite{ref:massera1949liapounoff,ref:Massera1956,ref:lefschetzlasalle1961,hahn1967stability}. Early generalizations for closed attractors, appear most notably in the work by Zubov \cite{ref:zubov1964methods}, see also \cite{ref:roxin1965stability,hahn1967stability}. Comments on the purely topological viewpoint can be found in \cite{ref:bhatiahajek2006local,ref:bhatia1970stability}, often under compactness assumptions.

Now, suppose that $A$ is a closed subset of a locally compact, separable, metric space $(X,d)$. If $A$ is uniformly asymptotically stable under some semi-dynamical systems $(X,\mathbb{R}_{\geq 0},\varphi)$, a converse Lyapunov theorem exists \cite[Thm. V.4.25]{ref:bhatia1970stability}. Note, although \cite{ref:bhatia1970stability} is only concerned with flows, the proof of their theorem works for semiflows as well. In particular, $A$ is not assumed to be invariant.

\begin{remark}[On topological stability]
\upshape{
One can readily define a purely topological notion of stability, replacing all metric neighbourhoods with arbitary neighbourhoods. However, as pointed out in earlier work by Auslander and Bhatia \cite{ref:Bhatiasaddle1970,ref:auslander1977filter}, one is less likely to be able to assert stability using a \textit{single continuous} Lyapunov function $V:X\to \mathbb{R}_{\geq 0}$. The reason being, the topology of $X$, or the structure of $\mathscr{F}_{\tau}$ for that matter, might be ``\textit{too complicated}'' to capture using $V$ (note that $\mathbb{R}_{\geq 0}$ is second countable while $X$ might even fail to be metrizable, in general). A metric space, on the other hand, is first countable and $x\mapsto d(x,A)$ allows for getting a grip on $\mathscr{F}_d$ using a single continuous function. To construct examples where a single $V$ indeed fails to exist, one should consider a non-metrizable space, as shown in \cite[p. 78-79]{ref:homotopietheorie} and \cite[Ex. III.9]{ref:jongeneelECC24}. We note that to properly interpret those results, one must consider so-called starting points \cite{ref:bhatiahajek2006local}.  
\exampleEnd
}
\end{remark} 

\begin{remark}[Metric and topological stability coincide when $A$ is compact]
\label{rem:metric:topo}
\upshape{
Despite the difficulties alluded to above, recall from our discussion on filters and in particular Lemma~\ref{lem:eps:nbhd:U} that when $A$ is compact, not merely closed, we can replace the metric neighbourhoods with standard open neighbourhoods, so that metric- and topological stability coincide. This is true, both for stability and attraction.
\exampleEnd
} 
\end{remark}

\begin{remark}[Stability is uniform when $A$ is compact]
\label{rem:unif:gas}
\upshape{
By Remark~\ref{rem:metric:topo}, we may focus on the metric definition of stability when $A$ is compact. In that case, stability and uniform stability are equivalent by \cite[Prop. V.4.2]{ref:bhatia1970stability} (see \cite[Rem. V.4.3]{ref:bhatia1970stability}, local compactness is not needed). Now, under the assumption that $(X,d)$ \textit{is} locally compact, one can show that the definition of uniform attraction through prolongations agrees with the topological definition \cite[Prop. V.1.2]{ref:bhatia1970stability}. Additionally, one can show that for compact attractors, asymptotic stability and uniform asymptotic stability are equivalent \cite[Thm. 1.16]{ref:bhatia1970stability}.
\exampleEnd
}
\end{remark}

Akin to Example~\ref{ex:open:nbhd:closed:1}, it is known that the uniform adjective does not come for free when $A$ is merely closed. 

\begin{example}[Non-uniform asymptotic stability]
\upshape{
Consider the following ODE on $\mathbb{R}^2$:
\begin{equation}
\label{equ:non:UGAS:ode}
\begin{pmatrix}
\dot{x} \\ \dot{y}
\end{pmatrix} = 
\begin{pmatrix}
0 \\ -y/ (1+x^2)
\end{pmatrix}.
\end{equation}
It readily follows that the flow corresponding to~\eqref{equ:non:UGAS:ode} can be defined through $\varphi^t( x, y ):=(x,e^{-a(x)t}y)$ with $\mathbb{R}\ni x \mapsto a(x) := 1/(1+x^2)$. Pick any $(x,y)\in \mathbb{R}^2$, then $\lim_{t\to +\infty}\varphi^t(x,y)\in A:=\{ (x,y)\in \mathbb{R}^2\,|\, y = 0\}$. In fact, $A$ is even uniformly stable. However, as $\lim_{|x|\to +\infty}a(x)=0$, the attraction is not uniform, that is, the $T(\varepsilon)>0$ from the definition is a function of $x$, without uniform bound.
\exampleEnd
}
\end{example}

In what follows, our running example will be a closed uniformly asymptotically stable attractor, hence we focus on this setting.

At last, given the stability notions we consider, we recall what can happen when we would allow for attractors that are not closed. For instance, in that case $A:=\mathbb{Q}$ is an attractor, with $B_u(A)=\mathbb{R}$ under the constant flow $(x,t)\mapsto \varphi^t(x):=x$ and the standard metric on $\mathbb{R}$. Clearly, these kind of pathologies should be excluded.

\section{A theorem by Wilson and its correction}
\label{sec:correct:Wilson}

The paper ``\textit{The Structure of the Level Surfaces of a Lyapunov Function}'' by Wilson \cite{ref:wilson1967structure} is rightfully celebrated, yet, it contains a flaw we must highlight here. In that work, the author claims that if $A$ is any uniformly asymptotically stable submanifold of some smooth manifold, under some smooth vector field, then the basin of attraction $B(A)$ is diffeomorphic to an open tubular neighbourhood of $A$ \cite[Thm. 3.4]{ref:wilson1967structure}. Note, a metric is used to define stability \cite[p. 327]{ref:wilson1967structure}, interestingly, to resolve precisely complications when $A$ is not compact.  

Lin, Yao and Cao recently corrected this theorem by Wilson, as the result is not true for \textit{any} submanifold. Compactness is key in their resolution. In particular, they prove the following. 

\begin{theorem}[{\cite[Thm. 1]{ref:lin2022wilson}} (Corrected version of {\cite[Thm. 3.4]{ref:wilson1967structure}})]
\label{thm:Wilson:correction}
The domain of attraction of a compact asymptotically stable submanifold $A$ in a finite-dimensional smooth manifold $X$ of an autonomous system \emph{[a smooth flow]} is diffeomorphic to the tubular neighborhood of $A$.
\end{theorem}

Then, to conclude on a flaw in Wilson's arguments, the authors provide a counterexample, showing that there a is closed but non-compact attractor such that Wilson's version of Theorem~\ref{thm:Wilson:correction} fails. 

Their counterexample will be one of our running examples. 

\begin{example}[{\cite[Ex. 22]{ref:lin2022wilson}} continued]
\label{ex:counterS1}
\upshape{
Let the metric space $(X,d)$ be given by $X:=\mathbb{R}^2\setminus \{(1,0)\}$ and $d(x,y):=\|x-y\|_2$. Let $\mathbb{S}^1 \hookrightarrow \mathbb{R}^2$ be the embedded unit circle and consider $A:=X\cap \mathbb{S}^1$. Hence, $A$ is closed. However, the sequence $\{\cos(1/k),\sin(1/k)\}_{k \geq 1}$ has no convergent subsequence in $A$, and hence $A$ is not compact (by the equivalence of compactness and sequential compactness on metric spaces). Now consider the negative gradient flow $\dot{x}=-\nabla f(x)$, for $X\ni x\mapsto f(x):=d(x,A)^2$, on $X$. It follows that $A$ is uniformly asymptotically stable with $B_u(A)=X\setminus \{(0,0)\}$. However, a tubular neighbourhood of $A$ is diffeomorphic to $A\times \mathbb{R}$ (which is a contractible set) and not to $B_u(A)$ (which is not a contractible set), see Figure~\ref{fig:S1closed} $(i)$, \textit{i.e.}, $A\, {\not\simeq_h}\, B_u(A)$. 

Without going into the details, $A$ and $B_u(A)$ are also not \textit{shape} equivalent, \textit{e.g.}, one may compute the $q$-th \u{C}ech-Alexander-Spanier cohomology group of $A$ and $B_u(A)$, denoted $\breve{H}^q(\cdot)$, via \cite[Prop. XIV.6.3]{ref:massey1991basicAT} ($B_u(A)$ has the homotopy type of the CW complex $\mathbb{S}^1 \vee \mathbb{S}^1$) and see, for instance, that
\begin{align*}
\breve{H}^1(A;\mathbb{Z})\cong H^1(A;\mathbb{Z}) \not \cong H^1(B_u(A);\mathbb{Z}) \cong \breve{H}^1(B_u(A);\mathbb{Z}). 
\end{align*}
For more details, consult \cite[Sec. XIV.6]{ref:massey1991basicAT}, \cite{ref:kapitanski2000shape} or \cite{ref:gobbino2001topological}.   
\exampleEnd 
}
\end{example}

To further illuminate the problem, recall Example~\ref{ex:open:nbhd:closed:1}. Now we elaborate on Example~\ref{ex:counterS1}.

    \begin{figure}
        \centering
        \includegraphics[scale=1]{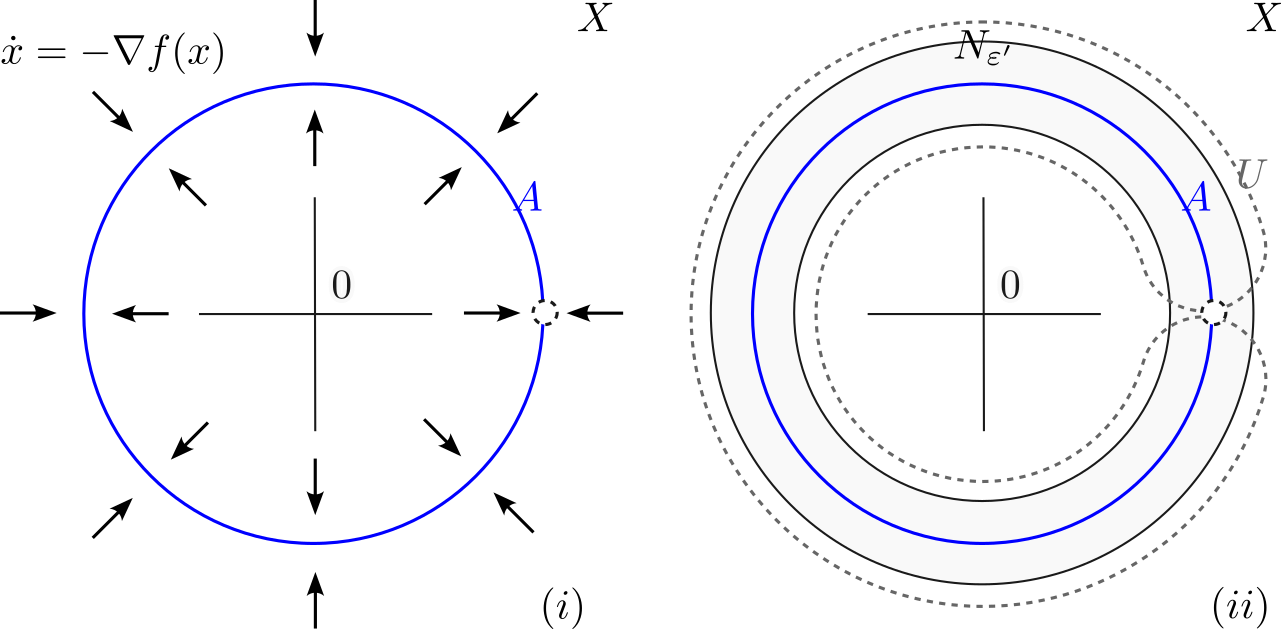}
        \caption{Example~\ref{ex:counterS1} and Example~\ref{ex:S1closed}. The set $A$ is a uniformly asymptotically stable attractor, yet, $A$ is not a strong deformation retract of $B_u(A)$. A (partial) reason being that $\mathscr{F}_{\tau}\not \leq \mathscr{F}_d$, \textit{e.g.}, $\not\exists V\in \mathscr{F}_d:V\subseteq U$.}
        \label{fig:S1closed}
    \end{figure}

\begin{example}[Example~\ref{ex:counterS1} continued: open neighbourhoods of closed subsets 2]
\label{ex:S1closed}
\upshape{
To compare with Example~\ref{ex:open:nbhd:closed:1}, now, the point $(1,0)$ takes up the role of ``$\infty$''. A choice of open neighbourhood $U$ of $A$ is as follows. Pick some $\varepsilon \in (0,1/2)$ and let $\mathbb{S}_{\varepsilon}^1:=\{(\cos(\theta),\sin(\theta))\in X\,|\, \theta\in[ \mathrm{asin}(\varepsilon),2\pi-\mathrm{asin}(\varepsilon)]\}$ and set $U:=N_{\varepsilon}(\mathbb{S}^1_{\varepsilon};d)$. It follows that there is no $\varepsilon'>0$ such that the open annulus $N_{\varepsilon'}(A;d)=\{ (r\cos(\theta),r\sin(\theta))\in X\,|\, r\in (1-\varepsilon',1+\varepsilon')),\, \theta\in [0,2\pi)\}$ is contained in $U$. See Figure~\ref{fig:S1closed} $(ii)$.         
\exampleEnd
} 
\end{example}

Despite the appeal of Example~\ref{ex:S1closed}, the existence of such a neighbourhood cannot be the \textit{sole} obstruction to $A \, {\not \simeq_h} \, B_u(A)$. Indeed, although Example~\ref{ex:open:nbhd:closed:1} provided us with a similar neighbourhood, for the dynamical system as shown in Figure~\ref{fig:xaxisclosed} $(i)$ we clearly have there that $\mathbb{R}^2=X=B_u(A)$ (strongly) deformation retracts onto $A$. In the next section we show that filters and cofibrations are the right tool to capture a topological mismatch as in Example~\ref{ex:counterS1}.

Now, one might expect that we can always simply restrict $B_u(A)$ to some open (forward invariant) subset $B'\subset B_u(A)$ such that $A\simeq_h B'$. By means of the example in Figure~\ref{fig:closedexamples} $(iii)$ one should observe that this is false in general. We also mention that completeness of the underlying space does not resolve the situation, consider the following examples: \cite[\S 4.2]{ref:lin2022wilson}, where a pinched cylinder functions as the domain of attraction; and \cite[Ex. III.8]{ref:jongeneelECC24}, where a Warsaw circle has an annular domain of attraction.

\section{Closed attractors}
\label{sec:main}
In previous work, we captured that when $A$ is a compact asymptotically stable attractor, $A$ is a strong deformation retract of $B(A)$ if and only if the inclusion map $\iota_A:A\hookrightarrow B(A)$ is a cofibration~\cite{ref:jongeneelECC24}. Interestingly, the attractor $A$ in Example~\ref{ex:counterS1} is a cofibration despite $A\,{\not\simeq_h}\, B_u(A)$, thus, the notion of a cofibration is not strong enough for our purposes, that is, for closed attractors that are possibly not compact. We show below, however, that $\mathscr{F}_d$-cofibrations are precisely the right tool to capture this.   

\subsection{Main results}
\begin{lemma}[A is a $\mathscr{F}_d$-weak deformation retract of $B_u(A)$]
\label{lem:Fd:weak:def:retract}
Let $A$ be a closed subset of a separable, locally compact metric space $(X,d)$. If $A$ is uniformly asymptotically stable, then $A$ is a $\mathscr{F}_d$-weak deformation retract of $B_u(A)$. 
\end{lemma}
\begin{proof}
Under the standing assumptions on $A$, there is a continuous Lyapunov function $V:B_u(A)\to \mathbb{R}_{\geq 0}$, with $V(x)=0 \iff x\in A$ and $V(\varphi^t(x))<V(x)$ on $B_u(A)\setminus A$, plus, there is a $\alpha \in \mathcal{K}_{\infty}$ such that $\alpha(d(x,A))\leq V(x)$ for all $x\in B_u(A)$ \cite[Thm. V.4.25]{ref:bhatia1970stability}. Additionally, there is some $\varepsilon>0$ such that $N_{\varepsilon}(A;d)\subseteq B_u(A)$ (by the definition of uniform attraction).  

Now, fix $c:=\alpha(\varepsilon/2)>0$ and define $T_c:B_u(A)\to \mathbb{R}_{\geq 0}$ through $T_c(x):=\inf\{t\geq 0:\varphi^t(x)\in V^{-1}([0,c])\}$. As $A$ is a uniform attractor, $T_c(x)<+\infty$ for all $x\in B_u(A)$. Also, for any $x\in B_u(A)\setminus V^{-1}([0,c])$, there is a $\delta>0$ such that $0<\delta<T_c(x)$ as $V^{-1}([0,c])\subset N_{\varepsilon}(A;d)\subseteq B_u(A)$. Then $T_c$ is continuous by the same line of arguments as in \cite[Thm. 3.6]{ref:kapitanski2000shape}. 

Then, given a $U\in \mathscr{F}_d$ contained in $B_u(A)$ (by the above there is at least one of them), we may assume that $N_{\varepsilon}(A;d)\subseteq U$ (otherwise rescale $\varepsilon$) and thus $V^{-1}([0,c])\subset U$. Hence, using the homotopy $B_u(A)\times [0,1]\ni (x,s)\mapsto H(x,s):=\varphi^{s\cdot T_c(x)}(x)$ we have established that any $U\in \mathscr{F}_d$ within $B_u(A)$ contains a strong deformation retract of $B_u(A)$ (namely, $V^{-1}([0,c])$ for an appropriate choice of $c>0$).   
\end{proof}

Note that we cannot strengthen Lemma~\ref{lem:Fd:weak:def:retract} to $\mathscr{F}_{\tau}$-weak deformation retracts, in general. This is impossible by, for instance, Example~\ref{ex:counterS1} and Example~\ref{ex:S1closed}. To be precise, in those examples $X$ is separable and locally compact, also, $B_u(A)$ contains $N_{1/2}(A;d)$, while $A$ is not a $\mathscr{F}_{\tau}$-weak deformation retract, as precisely the neighbourhood $U\in \mathscr{F}_{\tau}$, as in Figure~\ref{fig:S1closed}, contains no strong deformation retract of $B_u(A)$. Note, $U\notin \mathscr{F}_{d}$.  

\begin{lemma}[$\Leftarrow$~$\mathscr{F}_d$-cofibration]
\label{lem:cofib1}
Let $A$ be a closed subset of a separable, locally compact metric space $(X,d)$. Suppose that $A$ is uniformly asymptotically stable. If $\iota_A:A\hookrightarrow B_u(A)$ is a $\mathscr{F}_d$-cofibration, then, $A$ is a strong deformation retract of $B_u(A)$.
\end{lemma}
\begin{proof}
We follow \cite[Lem. III.4]{ref:jongeneelECC24}. As $\iota_A:A\hookrightarrow B_u(A)$ is a $\mathscr{F}_d$-cofibration, we know that $(B_u(A),A)$ is a $\mathscr{F}_d$-NDR pair by Theorem~\ref{thm:F:cofib:retract:ndr}. Now, we cannot directly conclude that $U$, as in Definition~\ref{def:F:NDR}, deformation retracts onto $A$, but we do have a \textit{deformation} $H:U\times [0,1]\to B_u(A)$, with $U\in \mathscr{F}_d$, that is, with $H$ as in Definition~\ref{def:F:NDR}. Then, as $A$ is a $\mathscr{F}_d$-weak deformation retract of $B_u(A)$, by Lemma~\ref{lem:Fd:weak:def:retract}, we know that $U$ contains a set $V\supseteq A$ such that $B_u(A)$ strongly deformation retracts onto $V$, that is, there is map $H_w:B_u(A) \times [0,1]\to B_u(A)$ such that $H_w(x,0)=x$ $\forall x\in B_u(A)$, $H_w(x,1)\in V$ $\forall x \in B_u(A)$ and $H_w(x,s)=x$ $\forall (x,s)\in V\times [0,1]$. Hence, the continuous map $\widetilde{H}:B_u(A)\times [0,1]\to B_u(A)$ defined through 
\begin{equation*}
    \widetilde{H}(x,s) := \begin{cases}
H_w(x,2s) \quad & s\in [0,1/2]\\
{H}\left(H_w(x,1),2s-1\right) \quad & s\in (1/2,1]
    \end{cases}
\end{equation*}
provides for the strong deformation retract of $B_u(A)$ onto $A$.  
\end{proof}

\begin{lemma}[$\Rightarrow$~$\mathscr{F}_d$-cofibration]
\label{lem:cofib2}
Let $A$ be a closed subset of a separable, locally compact metric space $(X,d)$. Suppose that $A$ is uniformly asymptotically stable. If $A$ is a strong deformation retract of $B_u(A)$, then, $\iota_A:A\hookrightarrow B_u(A)$ is a $\mathscr{F}_d$-cofibration. 
\end{lemma}
\begin{proof}
We appeal to Theorem~\ref{thm:F:cofib:retract:ndr}. As $A$ is a strong deformation retract of $B_u(A)$ by assumption, then, to conclude on $(B_u(A),A)$ being an $\mathscr{F}_d$-NDR pair, we need to construct the map $u:B_u(A)\to [0,1]$ from Definition~\ref{def:F:NDR}. As $A$ is uniformly asymptotically stable, we can appeal to the existence of a Lyapunov function. More precisely, under the standing assumptions on $A$, there is a continuous Lyapunov function $V:B_u(A)\to \mathbb{R}_{\geq 0}$, with $V(x)=0 \iff x\in A$ and $V(\varphi^t(x))<V(x)$ on $B_u(A)\setminus A$, plus, there is a $\beta \in \mathcal{K}_{\infty}$ such that $V(x)\leq \beta(d(x,A))$ for all $x\in B_u(A)$ \cite[Thm. V.4.25]{ref:bhatia1970stability}. Now, define $u$ through
    \begin{align*}
    B_u(A) \ni x \mapsto u(x) := \frac{V(x)}{1+V(x)}.
    \end{align*}
To show that $B_u(A)$ contains some neighbourhood $N_{\varepsilon}(A;d)$, see that $u(x)\leq \beta(d(x,A))$. Thus, for $\varepsilon := \beta^{-1}(1/2)>0$ we have that $u^{-1}([0,1))\in \mathscr{F}_d$.  
\end{proof}

\begin{theorem}[$\mathscr{F}_d$-cofibrations]
\label{thm:main:cofib}
Let $A$ be a closed subset of a separable, locally compact metric space $(X,d)$. Suppose that $A$ is uniformly asymptotically stable, then, $A$ is a strong deformation retract of $B_u(A)$ if and only if $\iota_A:A\hookrightarrow B_u(A)$ is a $\mathscr{F}_d$-cofibration.   
\end{theorem}
\begin{proof}
    Combine Lemma~\ref{lem:cofib1} and Lemma~\ref{lem:cofib2}. 
\end{proof}

\begin{example}[Retractions and convex spaces]
\label{ex:convex space}
\upshape{
The cofibration condition from Lemma~\ref{lem:cofib1} demands the existence of a \textit{deformation}, not necessarily a deformation retract. If $X\subseteq \mathbb{R}^n$ is convex, we have the following simple manifestation of an $\mathscr{F}_d$-NDR pair. Set $X\ni x \mapsto u(x):=\min(1,d(x,A))$, then if there is a retract $r:U\to A$, for $U:=N_{\varepsilon}(A)$ with $\varepsilon \in (0,1)$, we can define the homotopy $H:U\times [0,1]\to X$ through $H(x,s) := (1-s)x+sr(x)$. Note, in general, this is different from constructing a strong $\mathscr{F}_d$-\textit{neighbourhood} deformation retract as $H(U,[0,1])\subseteq U$ need not be true. As a constant map is a retract, see that if $A=\{\mathrm{pt}\}$ and $X=\mathbb{R}^n$, we directly recover \cite[Thm. 21]{ref:sontag2013mathematical}. 
\exampleEnd
}
\end{example}

\begin{example}[Example~\ref{ex:counterS1} continued: $A$ is not a strong deformation retract of $B_u(A)$]
\label{ex:counter:cofib}
\upshape{
We recall that in Example~\ref{ex:counterS1}, $A$ and $B_u(A)$ are not homotopy equivalent, despite the inclusion $\iota_A:A\hookrightarrow B_u(A)$ being a cofibration \textit{cf}. \cite{ref:jongeneelECC24} (\textit{e.g.}, using a tubular neighbourhood \cite[Ch. 6]{ref:Lee2}). Indeed, $A$ is \textit{not} compact. 

Now, we clarify \textit{via} Theorem~\ref{thm:main:cofib} why $A\,{\not\simeq_h}\,B_u(A)$. First, $A$ is not a strong $\mathscr{F}_d$-neighbourhood deformation retract of $B_u(A)$, as follows from the reasoning as put forth in Example~\ref{ex:counterS1} and Example~\ref{ex:S1closed}. By Theorem~\ref{thm:main:cofib}, this also shows that $\iota_A:A\hookrightarrow B_u(A)$ cannot be a $\mathscr{F}_d$-cofibration. However, to illustrate the utility of $\mathscr{F}_d$-cofibrations, we show that how some conclusions can be drawn in a more principled manner. Recall Remark~\ref{rem:homotopic:factor}, it follows that if $A\hookrightarrow B_u(A)$ is $\mathscr{F}_d$-cofibration, we must have
\begin{align*}
\pi_1(U) \overset{\iota_{U*}}{\longrightarrow} \pi_1(B_u(A)) = \pi_1(U) \overset{r_{*}}{\longrightarrow} \pi_1(A) \overset{\iota_{A*}}{\longrightarrow} \pi_1(B_u(A)) 
\end{align*}
but this factorization cannot hold as $\pi_1(A)=0$ (trivial fundamental group) while, by van Kampen's theorem, $\pi_1(B_u(A))\cong \pi_1(\mathbb{S}^1 \vee \mathbb{S}^1)\cong \mathbb{Z}*\mathbb{Z}\neq 0$, with $\iota_{U*} \pi_1(U) \not\cong 0$, for any $U\in \mathscr{F}_d$.

\exampleEnd
}
\end{example}

Then, to return to one of our closing comments in Section~\ref{sec:correct:Wilson}, the example in Figure~\ref{fig:xaxisclosed} $(i)$ clearly satisfies the conditions of Theorem~\ref{thm:main:cofib}, as it should.

\subsection{Relation to results for compact attractors}

Regarding Theorem~\ref{thm:main:cofib}, suppose that $A$ is compact, then, if $A\hookrightarrow B(A)$ is a $\mathscr{F}_d$-cofibration it is a $\mathscr{F}_{\tau}$-cofibration, \textit{i.e.}. a standard cofibration. Hence, Theorem~\ref{thm:main:cofib} generalizes \cite[Thm. III.6]{ref:jongeneelECC24}, which stated that when $A$ is a compact asymptotically stable attractor, $A$ is a strong deformation retract of $B(A)$ if and only if $A\hookrightarrow B(A)$ is a cofibration. That result allows us to conclude, for instance, that if $A$ is some compact, smooth embedded submanifold, then $A$ is a strong deformation retract of $B(A)$ \cite[Prop. III.10]{ref:jongeneelECC24}, see also \cite[Prop. 10]{ref:moulay2010topological}. Below, we discuss further insights.  

\subsubsection{Sets with positive reach}
As mentioned before, a convenient sufficient condition for $A\hookrightarrow X$ to be a cofibration, is that $A$ is a smooth embedded submanifold of $X$ \cite[Prop. III.10]{ref:jongeneelECC24}. We briefly illustrate in this subsection how a subclass of closed embedded submanifolds naturally connects to $\mathscr{F}_d$-cofibrations. To that end, we need to introduce the notion of \textit{reach}, as pioneered by Federer \cite[\S 4]{ref:federer1959curvature}.
To keep the presentation simple and avoid Riemannian geometry, we assume for the moment that $A$ is a subset of $\mathbb{R}^n$. 
In particular, let $A$ be a subset of a metric space $(\mathbb{R}^n,d)$, with $d(x,y):=\|x-y\|_2$, then the \textbf{\textit{reach}} of $A$ is defined through finding the largest metric neighbourhood of $A$ such that all its elements admit a unique projection onto $A$, that is, 
\begin{align}
\label{equ:reach}
\mathrm{reach}(A) := \sup\{r\geq 0\,|\, \forall x\in D_r(A;d)\, \exists !\, a^{\star} \in A: d(x,a^{\star})=d(x,A)\}. 
\end{align}
Sets with positive reach should be understood as generalizing convex sets in that $\mathrm{reach}(A)=+\infty \iff$ $A$ is closed and convex. Note also that defining reach through~\eqref{equ:reach} enforces a uniform bound.

Now, suppose that there is some $r>0$ such that $\mathrm{reach}(A)=r>0$ and let $\Pi_A:D_r(A;d)\to A$ be the corresponding projection operator, which is continuous \cite[Thm. 4.8.4]{ref:federer1959curvature}. Then, consider the map $H:D_r(A;d)\times [0,1]\to D_r(A;d)$ defined through $H(x,s) := \Pi_A(x) + (1-s)(x-\Pi_A(x))$. As $d(H(x,s),A)\leq (1-s)d(x,A)$, it follows that $H$ is a homotopy that captures that $D_r(A;d)$ strongly deformation retracts onto $A$ and thus $A\hookrightarrow \mathbb{R}^n$ is a $\mathscr{F}_d$-cofibration, but in particular, $A$ is a strong $\mathscr{F}_d$-neighbourhood deformation retract. To parametrize a path more naturally, one may consider gradient flows and geodesics in general.  



Now, when $A$ is compact (and a topological manifold), $\mathrm{reach}(A)>0$ if and only if $A$ is a $C^{1,1}$ embedded submanifold \cite{ref:federer1959curvature,ref:lytchak2004geometry,ref:lytchak2005almost}. 
In particular, $\mathrm{reach}(A)>0$ when $A$ is a compact $C^{\infty}$ (smooth) embedded submanifold.

When $A$ is not compact, one needs to control the Lipschitz moduli uniformly to enforce a strictly positive reach \cite{ref:rataj2017structure}. This is still an active topic of study, \textit{e.g.}, see \cite{ref:lieutier2024manifolds}. We provide an example.

\begin{example}[Curves with positive reach]
\upshape{
A simple $C^{1,1}$ curve $A$ is said to have the \textit{{quasi-arc property}} when for any $\varepsilon>0$ there is a $\delta>0$ such that $d(x_1,x_2)<\varepsilon$ whenever $x_1,x_2,x_3\in A$, $d(x_1,x_3)<\delta$ and the elements $x_1$ and $x_3$ are not contained in the same component of $A\setminus \{x_2\}$. 

Suppose that $A\subseteq \mathbb{R}^n$ is a closed, connected $1$-dimensional set. Then, if $A$ is not compact, $\mathrm{reach}(A)>0$ if and only if $A$ is a simple $C^{1,1}$ curve, with the quasi-arc property and being homeomorphic to either $\mathbb{R}_{\geq 0}$ or $\mathbb{R}$ \cite[Cor. 8.9]{ref:rataj2017structure}. An example is $A=\{(x,\sin(x))\,|\, x\in \mathbb{R}_{\geq 0}\}$  whereas a counterexample is, for instance, $A=\{(x,\sin(1/x))\,|\, x\in \mathbb{R}_{>0}\} \cup \{(0,0)\times [-1,1]\}$. Note that $A$ as in Example~\ref{ex:counterS1} is not a closed subset of $\mathbb{R}^2$. 
\exampleEnd
}
\end{example}

\subsubsection{On a relation to being able to linearize}
Let $M$ be a smooth manifold and let $A\subseteq M$ be a compact, invariant, globally asymptotically stable attractor under some flow $\varphi:M\times \mathbb{R}\to M$. Then, related to what we study, one might ask if this flow is \textit{\textbf{linearizable}}, that is, is there some matrix $B\in \mathbb{R}^{m\times m}$ and some continuous map $F:M\to \mathbb{R}^m$ such that $F \circ \varphi^t = e^{tB}\circ F$ for all $t\in \mathbb{R}$? Note, $F$ need not be a homeomorphism and in fact, $F$ is typically a topological embedding. This relates to what is called a Koopman linearization. It turns out that for the setting as sketched above, yet with $\varphi$ being a \textit{smooth} flow and $F$ being a \textit{smooth} embedding, $A$ must be a smooth embedded submanifold of $M$  \cite[Thm. 4]{ref:kvalheim2023linearizability}. However, this means that $A\hookrightarrow M$ must be a cofibration and thus, $A\simeq_h M$. Hence, homotopy equivalence is necessary for such a linearization to exist. 
\begin{corollary}
Let $M$ be a smooth manifold and let $A\subseteq M$ be a compact, invariant, globally asymptotically stable attractor under some smooth flow $\varphi:M\times \mathbb{R}\to M$, then, $\varphi$ is linearizable, by a smooth embedding, only if $A\simeq_h M$. 
\end{corollary}
Note, when $A$ is closed, but not compact, this fails to be true, as visualized by Figure~\ref{fig:closedexamples} $(ii)$, \textit{i.e.}, $X=\mathbb{R}^2\setminus \{0\}$, $A=\{(x_1,x_2)\in X\,|\,x_2=0\}$ and $B_u(A)=X$ while $A\simeq_h \mathbb{S}^0$ and $B_u(A)\simeq_h \mathbb{S}^1$, nevertheless, if we denote the corresponding flow by $\varphi$, we have that $F\circ \varphi^t = e^{Bt}\circ F$ for $F:X\hookrightarrow \mathbb{R}^2$ and $B=\mathrm{diag}(0,-1)\in \mathbb{R}^{2\times 2}$, thus a linearizing $F$ and $B$ do exist, trivially, despite $A\, {\not \simeq_h}\, B_u(A)=:M$.

\section{Discussion, conclusion and future work}
\label{sec:conclusion}

In this work we have characterized when $A$ is a strong deformation retract of $B_u(A)$ (Theorem~\ref{thm:main:cofib}) through an adaptation of cofibrations (Theorem~\ref{thm:F:cofib:retract:ndr}). 

We were motivated to study homotopy questions in the context of dynamical systems to understand limitations of (continuous) feedback. More specific, suppose we have a control system $\dot{x}=f(x,u)$ on a metric space $(X,d)$, where $u$ denotes the input, such that any admissible feedback $x\mapsto \mu(x)$ is such that the closed-loop system $\dot{x}=F(x):=f(x,\mu(x))$ results in a global semiflow (we are deliberately vague about the precise input structure and feedback regularity as this is not relevant for what follows). Then, if our goal is to globally uniformly asymptotically stabilize some closed set $A\subseteq X$ by means of some feedback $x\mapsto \mu(x)$, we must comply with the constraint $A\simeq_h X$ in case $A\hookrightarrow X$ is a $\mathscr{F}_d$-cofibration.  

Typically, $A\not\simeq_h X$, and obstructions (constraints) of this form motivate the introduction of discontinuities in our feedback, \textit{e.g.}, to globally asymptotically stabilize a point on the circle we need to ``cut'' it. This is the area of hybrid feedback control, \textit{e.g.}, see \cite{ref:sanfelice2020hybrid}. Theorem~\ref{thm:main:cofib} allows us to comment, with some ease and without relying on compactness, on the \textit{topological perplexity} \cite{ref:baryshnikov2023topological} of the global stabilization problem, that is, from a feedback stabilization perspective, one can describe the topological mismatch between $A$ and $X$, one can describe the complexity of the necessary ``cuts". 

Then, besides the search for further manifestations of $\mathscr{F}_d$-cofibrations (\textit{e.g.}, consider the \textit{uniform} tubular neighbourhood thereom in \cite[Thm. 2.33]{ref:eldering2013normally}), plus the development of numerical and discrete counterparts, we have identified several other questions and directions of interest. 
\begin{enumerate}[(i)]
\item Is there a weakest set of assumptions, in a topological sense, on the pair $(X,A)$ to have an appropriate (continuous) converse Lyapunov theory? 
\item Can results of this form be inferred from a categorical approach to Lyapunov theory, possibly allowing for a unified (regularity) study? We are inspired here by \cite{ref:ames2025categorical,ref:ames2025categoricalii}. 
\item Can results regarding homotopies of vector fields (and semiflows) that stabilize compact attractors also be extended to similar results for closed attractors? For references, see \cite{ref:reineck1991continuation,ref:kvalheim2022obstructions,ref:JongeneelSchwan2024TAC,ref:jongeneel2024hGAS,ref:kvalheim2025differential}.    
\item In general, we believe that Auslander's work on stability through filters \cite{ref:auslander1977filter} has more to offer, here we are encouraged by simple observations in Section~\ref{sec:retract:and:filters}, \textit{e.g.}, Lemma~\ref{lem:strong:def:retract:filter}. What can be said in general? Subject to converse Lyapunov theory, we believe that Theorem~\ref{thm:main:cofib} can be generalized to $\mathscr{F}$-cofibrations. 
\end{enumerate}

\pagestyle{basicstyle}
\addcontentsline{toc}{section}{Bibliography}
\subsection*{Bibliography}
\printbibliography[heading=none]


\end{document}